\input{epsf}

\documentclass[11pt,letterpaper]{article}

\usepackage[hang,perpage]{footmisc}

\usepackage{amsmath}
\usepackage{amsthm}
\usepackage{amssymb}
\usepackage{stmaryrd}

\usepackage{graphicx}
\graphicspath{{../../figs/}}

\usepackage{latexsym}
\usepackage{times}
\usepackage[varumlaut]{yfonts}
\usepackage[usenames]{color}
\usepackage{braket}
\usepackage{mathabx}
\usepackage{textcomp}
\usepackage{makeidx}
\usepackage{tocbibind}
\usepackage[bookmarks=true,backref=true,colorlinks=true,linkcolor=darkolivegreen,
citecolor=darkolivegreen,urlcolor=darkolivegreen]{hyperref}
\usepackage[T1]{fontenc}
\usepackage{mathrsfs}

\numberwithin{equation}{subsection}
\allowdisplaybreaks

\theoremstyle{definition}
\newtheorem{ass}{Assumption}[subsection]
\newtheorem{theorem}[ass]{Theorem}
\newtheorem{lemma}[ass]{Lemma}
\newtheorem{definition}[ass]{Definition}
\newtheorem{corollary}[ass]{Corollary}

\newtheorem{rem}[ass]{Remark}

\makeindex

\newcommand{\captionfonts}{\footnotesize}
\makeatletter  
\long\def\@makecaption#1#2{%
  \vskip\abovecaptionskip
  \sbox\@tempboxa{{\captionfonts #1: #2}}%
  \ifdim \wd\@tempboxa >\hsize
    {\captionfonts #1: #2\par}
  \else
    \hbox to\hsize{\hfil\box\@tempboxa\hfil}%
  \fi
  \vskip\belowcaptionskip}
\makeatother   
\definecolor{darkolivegreen}{rgb}{0.333333, 0.419608, 0.1843140}

\makeatother

\begin{document}


\title{On the characterization of asymptotic cases of the diffusion 
equation with rough coefficients and applications to preconditioning}

\author{Burak Aksoylu \\
Louisiana State University \\ 
Department of Mathematics \& \\
Center for Computation and Technology\\
Baton Rouge, LA 70803, USA
\and
Horst R. Beyer \\
Louisiana State University \\
Center for Computation and Technology \\
Baton Rouge, LA 70803, USA}

\date{\today}                                     

\maketitle

\begin{abstract}
We consider the diffusion equation in the setting of operator theory.
In particular, we study the characterization of the limit of the
diffusion operator for diffusivities approaching zero on a
subdomain $\Omega_1$ of the domain of integration of $\Omega$.  We
generalize Lions' results to covering the case of diffusivities which
are piecewise $C^1$ up to the boundary of $\Omega_1$ and $\Omega_2$,
where $\Omega_2 := \Omega \, \setminus \, \overline{\Omega}_1$ instead
of piecewise constant coefficients.  In addition, we extend both
Lions' and our previous results by providing the strong convergence of
$\left(A_{\bar{p}_\nu}^{-1}\right)_{\nu \in \mathbb{N}^\ast},$ for a
monotonically decreasing sequence of diffusivities
$\left( \bar{p}_\nu \right)_{\nu \in \mathbb{N}^\ast}$.
\end{abstract}


\section{Introduction}
The diffusion equation
\begin{equation} \label{diffusioneq}
\frac{\partial u}{\partial t} = \textrm{div}  
\left(\, p \, \textrm{grad} \, u \right) + f  
\end{equation}
describes general diffusion processes, including the propagation of
heat, and flows through porous media. Here $u$ is the density of the
diffusing material, $p$ is the diffusivity of the material, and the
function $f$ describes the distribution of `sources' and `sinks'.  The
usage of $\bar{p}:= 1/p$ provides a convenient framework to study
asymptotic cases where diffusivity approaches zero on an open subset
of non-zero measure.  Therefore, our definitions will be based on
$\bar{p}$. This paper focuses on stationary solutions of 
(\ref{diffusioneq}) satisfying
\begin{equation} \label{diffusioneqstat}
- \textrm{div}  
\left(\, \left(1/\bar{p}\right) \, \textrm{grad} \, u \right) = f  \, \, .
\end{equation}
For instance, the fictitious domain method and composite materials are
sources of rough coefficients; see the references in~\cite{KnWi:2003}.
Important current applications deal with composite materials whose
components have nearly constant diffusivity, but vary by several
orders of magnitude.  In composite material applications, it is quite
common to idealize the diffusivity by a piecewise constant function
and also to consider limits where the values of that function approach
zero or infinity in parts of the material.
\newline
\linebreak
For the treatment of these questions, we use methods from 
operator theory. For this, we use a common approach to 
give (\ref{diffusioneq}) a well-defined meaning that, in a 
first step, represents the diffusion 
operator 
\begin{equation} \label{diffusionoperator}
- \textrm{div} \, \left(1/\bar{p}\right) \, \textrm{grad}  
\end{equation}
as a densely-defined positive self-adjoint linear operator
$A_{\bar{p}}$ in $L_{\mathbb{C}}^2(\Omega)$. As a result,
(\ref{diffusioneqstat}) is represented by the equation
\begin{equation} \label{diffusionEquation2}
A_{\bar{p}} u = f \, \, , 
\end{equation}
where $f$ is an element of the Hilbert space, and $u$ is 
from the domain, $D(A_{\bar{p}})$, of $A_{\bar{p}}$.
\newline
\linebreak
In our previous paper~\cite{AkBe2008}, we treat diffusivities from the
class ${\cal L}$ consisting of $p \in L^{\infty}(\Omega)$ that are
defined almost everywhere $\geq \varepsilon$ on $\Omega$ for some
$\varepsilon >0$, where $\Omega \subset {\mathbb{R}}^n$, $n \in
{\mathbb{N}}^{*}$, is some non-empty open subset. By use of Dirichlet
boundary conditions, $\bar{p} \in {\cal L}$ induces a densely-defined,
linear, self-adjoint, strictly positive operator $A_{\bar{p}}$ in
$L^2_{\mathbb{C}}(\Omega)$. By assuming
a weak notion of convergence in ${\cal L}$, we showed that the maps
$\mathcal{S}$ and $\mathcal{T}$ defined by 
\begin{eqnarray} \label{soluMapS}
\mathcal{S}(\bar{p}) & := & A_{\bar{p}}^{-1} \, , \\
\label{soluMapT}
\mathcal{T}(\bar{p}) & := & -\left(1/\bar{p}\right) \, 
\nabla A_{\bar{p}}^{-1} \,,
\end{eqnarray}
for every $\bar{p} \in \mathcal{L}$ are strongly sequentially continuous.
\newline
\linebreak
For the case $n=1$ and bounded open intervals of ${\mathbb{R}}$, we
were able to show stronger results that include also the asymptotic
cases, except that where the asymptotic `diffusivity' is almost
everywhere infinite on the whole interval.  We showed
that $\mathcal{S}$ and $\mathcal{T}$ have unique extensions to
sequentially continuous maps $\hat{\mathcal{S}}$ and
$\hat{\mathcal{T}}$ in the operator norm on the set of a.e. positive
elements of $L^{\infty}(\Omega) \setminus \{0\}$. In addition, an
explicit estimate of the convergence behaviour of the maps is given,
\begin{equation} \label{perturbationExpansion}
\hat{\mathcal{S}}(\bar{p}) = \hat{\mathcal{S}}(\bar{p}_\infty) + 
\mathcal{O}(\| \bar{p} - \bar{p}_\infty \|_1).
\end{equation}
Furthermore, we explicitly calculated $\hat{\mathcal{S}}$ and
$\hat{\mathcal{T}}$. The knowledge of $\hat{\mathcal{S}}$ and
$\hat{\mathcal{T}}$ for asymptotic $p$ is essential for the purpose of
preconditioning. Since $\hat{\mathcal{S}}$ maintains continuity on
$\partial \mathcal{L}$, the boundary value can be used as the dominant
factor in a perturbation expansion for $\mathcal{S}(\bar{p})$ for
$\bar{p} \in \mathcal{L}$.
By rewriting \eqref{perturbationExpansion}, we arrive at an expression
for a preconditioned operator:
$$
\begin{array}{lllll}
A_{\bar{p}}^{-1} & = & A_{\bar{p}_\infty}^{-1}  & + &
\mathcal{O}(\| \bar{p} - \bar{p}_\infty \|_1) \, ,\\
A_{\bar{p}_\infty}^{-1} A_{\bar{p}} & = & I & + & 
\mathcal{O}(\| \bar{p} - \bar{p}_\infty \|_1).
\end{array}
$$
\begin{figure}[htbp]
\centering{
\includegraphics[width=3.5cm]{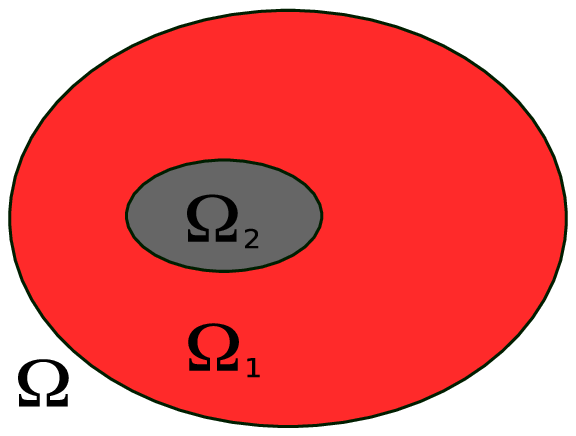}
\hspace{1cm}
\includegraphics[width=3.5cm]{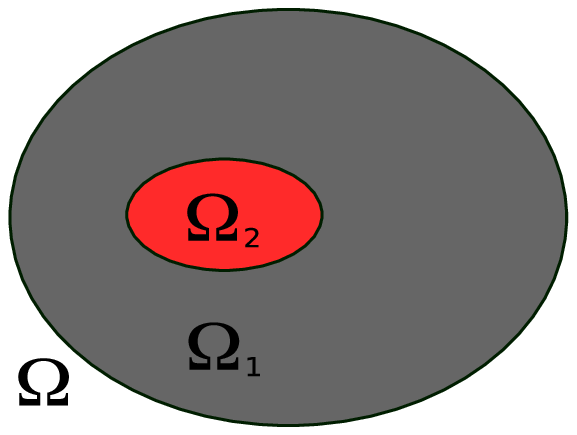}
}
\caption{Red and gray color indicate high and low diffusivity values, 
respectively. (Left) Lions' subdomain configuration. (Right) The configuration
with diffusivity values swapped. \label{fig:domain1}}
\end{figure}
For preconditioning purposes, in this paper, we study the boundary
behaviour of $\mathcal{S}$ for $n > 1$. We prove the strong
convergence of $\mathcal{S}(\bar{p}_\nu)$ for any monotonically
decreasing $\left(\bar{p}_\nu\right)_{\nu \in \mathbb{N}}$. We also
characterize the associated limits for particular cases.  The
establishment of these results is the goal of the present article.  As
expected, the limits are structurally simpler. Therefore, utilizing
the limits as preconditioners should lead to computationally feasible
preconditioning.  One such approach was taken by the first author
in~\cite{AGKS:2007}.  For showing effectiveness of the proposed
preconditioner, one utilizes spectral equivalences. For the derivation
of such equivalences, operator theory provides the natural framework.
These questions are the subject for further study.

\section{Previous results and our improvements}

The treatment of the diffusion equation with piecewise constant
discontinuous coefficient has been pioneered by
J. L. Lions~\cite{lions}.  In his lecture notes, he considers the
limit of the solution of (\ref{diffusioneqstat}) where the limit is
associated to a one-parameter family of piecewise constant diffusivities
$\left( \bar{p}_{\varepsilon} \right)_{\varepsilon \, \in \,
(0,\infty)}$ approaching zero on a subdomain $\Omega_1$ of an open
subset $\Omega$ of $\mathbb{R}^n,~n \geq 1$.  In this, the boundary of
$\Omega_1$ intersects that of $\Omega$; see the left of 
Figure~\ref{fig:domain1}.
\newline 
\linebreak
Using a first order formulation of the diffusion operator, a similar
piecewise constant one-parametric approach was used
in~\cite{bakhvalov1,knyazev1}, but with diffusivities approaching
infinity on a subdomain; see the right of Figure~\ref{fig:domain1}.
By a simple scaling argument, it can be seen that the results based on
such subdomain configuration can be reproduced from the Lions'
configuration and vice versa.
\newline \linebreak
In addition to our aforementioned one-dimensional results, in the
previous paper~\cite{AkBe2008}, by assuming a weak notion of
convergence in ${\cal L}$ for $n \geq 2$, we showed that the solution
maps $\mathcal{S}$ and $\mathcal{T}$, defined in \eqref{soluMapS} and
\eqref{soluMapT} respectively, are strongly sequentially continuous. 
\newline 
\linebreak
The basis of Lions' results provides an abstract lemma which derives
a Laurent expansion for $\xi_{\varepsilon}$ in terms of $\varepsilon$
satisfying the equation
\begin{equation} \label{sesquilinearEquation}
s_1({\xi},\xi_{\varepsilon})  + \varepsilon \, s_2({\xi}, \xi_{\varepsilon}) 
= \braket{{\xi}|\eta}
\end{equation}
for every $\xi \in X$. Here, $\eta \in X$ is given and $s_1, s_2$ are
prescribed sesquilinear forms on the abstract Hilbert space $X$
satisfying certain conditions. The weak formulation
of \eqref{diffusioneqstat} corresponding to $\bar{p}_{\varepsilon}$
leads to this class of problems.  Lions sketched the proof, only. For
the convenience of the reader, this lemma is given in the appendix
along with a full proof; see Lemma~\ref{basiclemmalions} and
Lemma~\ref{lemma:LionsAbstract}.  In addition, Lions sketched the
application of this lemma to the diffusion equation with piecewise
constant coefficients as an example.  Here, we extend Lions'
results in various directions.  In particular, we consider strong
solutions of the operator equation instead of Lions' weak solutions.
Note that the establishment of these results is based solely on the
foundation provided in our preceding paper~\cite{AkBe2008}.

\begin{itemize}
\item[i)] 
We generalize Lions' example to a theorem covering the case of
diffusivities which are piecewise $C^1$ up to the boundary of
$\Omega_1$ and $\Omega_2$, where 
$\Omega_1 := \Omega \, \setminus \, \overline{\Omega}_2$ 
instead of piecewise constant coefficients.  Note that for this,
as is also the case in Lions' result, the source function $f$ is
required to be an element of $W^1_{0,{\mathbb{C}}}({\Omega})$ which
incorporates a regularity condition and a homogeneous boundary
condition.  

\item[ii)] 
In addition, we extend both Lions' and our previous results by 
providing the strong convergence of
$$
\left( \mathcal{S}(\bar{p}_\nu) \right)_{\nu \in \mathbb{N}^\ast}
= \left(A_{\bar{p}_\nu}^{-1}\right)_{\nu \in \mathbb{N}^\ast},
$$ for a
monotonically decreasing sequence $\bar{p}_1, \bar{p}_2, \ldots$ in
$\mathcal{L}$; see Theorem~\ref{thm:variationalFramework}.  The
coefficients in Lions' case, i.e., one parametric piecewise constant
coefficients, automatically lead to a particular case of a
monotonically decreasing sequence of diffusivities.  Differently from
Lions' case, our construction does not require a particular
configuration of subdomains.  On the other hand, differently to Lions,
our theorem does not give a characterization of the corresponding the
strong limit.  Also, Lions shows convergence in the stronger
$W^1$-norm as opposed to convergence in the $L^2$-norm, here.
\end{itemize}

\section{Preliminaries}

\label{prerequisites}

\begin{definition} \label{weaksolutions} ({\bf Weak solutions})
Let $X$ be a non-trivial complex Hilbert space, $A : D(A) \rightarrow 
X$ be a densely-defined, linear, self-adjoint and strictly 
positive operator in $X$. For $\eta \in X$, we call 
$\xi \in D(A^{1/2})$ 
a weak solution of the equation 
\begin{equation} \label{invertA}
A \xi = \eta 
\end{equation}
if 
\begin{equation*}
\braket{A^{1/2} \xi|A^{1/2} \xi^{\prime}} =
\braket{\eta | \xi^{\prime}}
\end{equation*} 
for every $\xi^{\prime} \in D(A^{1/2})$. 
\end{definition}

\begin{rem}
We note that
the `strong' solution of the equation (\ref{invertA}),    
$\xi := A^{-1} \eta$, is also a weak solution of that 
equation. In addition, by the bijectivity of $A^{1/2}$, 
it follows the uniqueness of a weak solution. Hence 
$\xi \in D(A^{1/2})$ is a weak solution of the equation 
(\ref{invertA}) if and only if it is a strong solution
of (\ref{invertA}), i.e., if and only if 
$\xi \in D(A)$ and $A \xi = \eta$.  
\end{rem}

We define the diffusion operator as operator in
$L^{2}_{\mathbb{C}}({\Omega})$ and give basic properties.

Diffusion operators corresponding to diffusivities from 
the following large subset ${\cal L}$ of $L^{\infty}(\Omega)$ 
will turn out to be densely-defined, linear, self-adjoint operators. 

\begin{definition}
We define the subset ${\cal L}$ 
of $L^{\infty}(\Omega)$ to consist 
of those elements $\bar{p}$ for which there are real 
$C_1, C_2$ satisfying $C_2 \geq C_1 > 0$ and such that
$C_1 \leq \bar{p} \leq C_2$ a.e. on $\Omega$. Note that the last 
also implies that $1 / \bar{p} \in {\cal L}$ 
and in particular that  
$1 / C_2 \leq 1 / \bar{p} \leq 1/ C_1$ a.e. on $\Omega$. 
\end{definition}

\begin{definition} \label{defA} For $\bar{p} \in \mathcal{L}$, we define the 
linear operator $A : D(A) \rightarrow L^{2}_{\mathbb{C}}({\Omega})$ 
in $L^{2}_{\mathbb{C}}({\Omega})$ by 
\begin{equation*}
D(A) := \{ u \in W^1_{0,{\mathbb{C}}}({\Omega}): 
(1 / \bar{p}) {\nabla_{w}} u \in D({\nabla_{\!0}}^{*}) \}
\end{equation*}
and 
\begin{equation*}
A u := {\nabla_{\!0}}^{*} (1 / \bar{p}) \, {\nabla_{w}} u
\end{equation*} 
for every $u \in D(A).$
\end{definition}

\begin{theorem} \label{secondorderoperator}
Let $\bar{p} \in {\cal L}$.  Then $A$ is a densely-defined, linear,
self-adjoint operator in $L^{2}_{\mathbb{C}}({\Omega})$.
\end{theorem}

\begin{proof}
See \cite[Theorem 4.0.9]{AkBe2008}.
\end{proof}
\begin{theorem} \label{regularcases}
Let $\Omega$ be in addition bounded with a boundary of class 
$C^2$ and $\bar{p} \in C^1(\,\bar{\Omega},{\mathbb{R}})$. 
Then 
\begin{equation} \label{regularcases1}
D(A) = W^1_{0,{\mathbb{C}}}(\Omega) \cap  
W^2_{{\mathbb{C}}}(\Omega) \, \, .
\end{equation}
\end{theorem}

\begin{proof}
The statement is a simple consequence of elliptic regularity.
\end{proof}

The following statement will be used in the proof of
Theorem~\ref{thm:variationalFramework}.  Note that the domain of the
quadratic form $q_{\eta}$ is given by $D(A^{1/2})$, which is generally
larger than $D(A)$. The same result holds if the domain of $q_{\eta}$ is
restricted to $D(A)$. However, $D(A)$ depends heavily on the diffusivity,
whereas, according to \cite[Lemma 5.0.19]{AkBe2008}, 
$D(A^{1/2}) = W^1_{0,\mathbb{C}}(\Omega)$.

\begin{theorem} \label{variationalformulation} ({\bf Variational 
formulation})
Let $X$ be a non-trivial complex Hilbert space, $A : D(A) \rightarrow 
X$ be a densely-defined, linear, self-adjoint and strictly 
positive operator in $X$. Further, 
let $\eta \in X$ and 
$q_{\eta} : D(A^{1/2}) \rightarrow {\mathbb{R}}$ be defined by 
\begin{equation*}
q_{\eta}(\xi) := \braket{A^{1/2} \xi | A^{1/2} \xi} 
- \braket{\eta | \xi}
- \braket{\xi | \eta}
\end{equation*}
for every $\xi \in D(A^{1/2})$. Then $q_{\eta}$ assumes a unique 
minimum, of value 
\begin{equation*}
- \braket{\eta|A^{-1}\eta} \, \, , 
\end{equation*}
at $\xi = A^{-1} \eta$.
\end{theorem}

\subsection{Variational formulation}

In the following, we show the strong convergence of
$\left(A_{\bar{p}_\nu}^{-1}\right)_{\nu \in \mathbb{N}^\ast}$ for a
monotonically decreasing sequence $\bar{p}_1, \bar{p}_2, \ldots$ in
$\mathcal{L}$. 

\begin{lemma} \label{thm:variationalFramework}
Let ${\bar{p}}_{1}, {\bar{p}}_{2}, \dots $ be a monotonically 
decreasing, i.e., such that for every $\nu \in {\mathbb{N}}^{*}$
the inequality 
${\bar{p}}_{\nu + 1}(x) \leq {\bar{p}}_{\nu}(x)$ holds for almost 
all  
$x \in \Omega$, 
sequence in ${\cal L}$. 
In addition, let 
$A_1,A_2,\dots$ be the associated 
sequence of self-adjoint
linear operators. Then the sequence $A_1^{-1},A_2^{-1},\dots$
is strongly convergent to a positive bounded self-adjoint linear 
operator on $L^2_{\mathbb{C}}(\Omega)$.
\end{lemma}

\begin{proof}
For this, let $\nu \in {\mathbb{N}}^{*}$. Further, 
let $f \in L^2_{\mathbb{C}}(\Omega)$ and 
$q_{\nu, f} : W^{1}_{0,{\mathbb{C}}}({\Omega}) 
\rightarrow {\mathbb{R}}$ be defined by 
\begin{align*}
q_{\nu,f}(u) := & \braket{A_{\nu}^{1/2} u | A_{\nu}^{1/2} u}_{2} 
- \braket{f | u}_{2}
- \braket{u | f}_{2} \\
= & \braket{\,{\nabla_{w}} u \, | \, (1 / {\bar{p}}_{\nu} ) 
{\nabla_{w}} u \, }_{2,n} - \braket{f | u}_{2}
- \braket{u | f}_{2}
\end{align*}
for every $u \in  W^{1}_{0,{\mathbb{C}}}({\Omega}) $.
\begin{equation*}
\|A_{\nu}^{1/2} f\|_{2}^2 = 
\braket{\,{\nabla_{w}} f \, | \, ( 1 / {\bar{p}}_{\nu} ) 
{\nabla_{w}} f \, }_{2,n} 
\end{equation*}
for every $f \in 
W^{1}_{0,{\mathbb{C}}}({\Omega})$. According 
to Theorem~\ref{variationalformulation}, 
$q_{\nu,f}$ assumes a unique 
minimum, of value 
\begin{equation*}
- \braket{f|A_{\nu}^{-1}f}_{2} \, \, , 
\end{equation*}
at $u_{\nu} := A_{\nu}^{-1} f$. As a consequence, 
since ${\bar{p}}_{1}, {\bar{p}}_{2},\dots $ is monotonically 
decreasing, it follows that 
\begin{equation*}
q_{\nu + 1,f}(u) \geq q_{\nu,f}(u) 
\end{equation*}
for every $u \in  W^{1}_{0,{\mathbb{C}}}({\Omega})$
and hence that 
\begin{equation*} 
q_{\nu + 1,f}(u_{\nu + 1}) \geq q_{\nu,f}(u_{\nu + 1}) \geq 
q_{\nu,f}(u_{\nu}) \, \, . 
\end{equation*}
Hence it follows that 
\begin{equation*}
\braket{f|A_{\nu + 1}^{-1}f}_{2} \leq \braket{f|A_{\nu}^{-1}f}_{2} 
\, \, .
\end{equation*}
From the last, it follows that $A_{1}^{-1},A_{2}^{-1},\dots$
is a monotonically decreasing sequence of positive bounded 
self-adjoint operators on $L^2_{\mathbb{C}}(\Omega)$ 
and as such strongly convergent to
a positive bounded 
self-adjoint operator on $L^2_{\mathbb{C}}(\Omega)$. 
\end{proof}

\subsection{Generalization of Lions' Lemma}
We provide the structures satisfying the assumptions of
Lemma~\ref{lemma:LionsAbstract} for treating the diffusion
equation~\eqref{diffusionEquation2}.

\begin{theorem} \label{thm:generalizationLions}
Let $\Omega$ be a non-void bounded open subset of $\mathbb{R}^n$ 
with boundary of class $C^2$, 
$\Omega_2 \subset {\mathbb{R}}^n$ be such that 
${\bar{\Omega}}_2 \subset \Omega$ and with boundary of 
class $C^2$. We define the closed subspace 
$X_1$ of $W^1_{0,{\mathbb{C}}}({\Omega})$ by the range of the 
isometric imbedding $\iota$ of 
$W^{1}_{0, \mathbb{C}}({\Omega_2})$ into 
$W^{1}_{0, \mathbb{C}}({\Omega})$ given by $\iota(f) := \hat{f}$
for every $f \in 
W^{1}_{0, \mathbb{C}}({\Omega_2})$, where
\begin{equation*}
{\hat{f}}(x) := \begin{cases}
f(x) &\text{if $x \in {\Omega_2}$} \\
\, \, 0 &\text{if $x \in \Omega \, \setminus \, {\Omega_2}$}  \, \,. 
\end{cases}
\end{equation*}
In addition, let $p_1$ and $p_2$ be a.e. positive 
elements of $L^{\infty}_{\mathbb{C}}(\Omega)$ that vanish almost 
everywhere on $\Omega_2$ and $\Omega_1$, respectively, satisfy 
$p_1|{\Omega_1} \in C^1({\bar{\Omega}}_1,{\mathbb{R}})$, 
$p_2|{\Omega_2} \in C^1({\bar{\Omega}}_2,{\mathbb{R}})$ and 
for which there are $\alpha_1,\alpha_2 > 0$ such that 
$p_j \geq \alpha_j$ almost everywhere on ${\Omega}_j$, 
$j \in \{1,2\}$.
Finally, let 
$s_{j} : (W^1_{0,{\mathbb{C}}}({\Omega}))^2 \rightarrow {\mathbb{C}}$
be defined by 
\begin{equation*}
s_{j}(f,g) := \braket{\,{\nabla_{w}} f \, | \, p_j 
{\nabla_{w}} g \, }_{2,n}
\end{equation*}
for all $f,g \in  W^1_{0,{\mathbb{C}}}({\Omega})$ and 
$j \in \{1,2\}$. Then $W^1_{0,{\mathbb{C}}}({\Omega})$, $X_1$, 
$s_1$, $s_2$ satisfy the General~Assumption~\ref{lionsass}.
\end{theorem}

\begin{proof}
By
\begin{equation*}
\|f\| := \left(\,\sum_{k=1}^n \|\partial^{k} f \|_2^2 \right)^{1/2} 
\end{equation*}
for all $f \in W^{1}_{0, \mathbb{C}}(\Omega)$, there is defined 
a norm $\|\,\|$ on 
$W^{1}_{0, \mathbb{C}}(\Omega)$ that is equivalent to 
$\vvvert \, \vvvert_{1}$. Therefore, without restriction, we 
can assume in the following that $W^{1}_{0, \mathbb{C}}(\Omega)$ 
is equipped with the norm  $\|\,\|$. Further, by use of 
the inequalities
\begin{align*}
& |s_{j}(f,g)| = |\braket{\,{\nabla_{w}} f \, | \, p_j 
{\nabla_{w}} g \, }_{2,n}| = \left|
\sum_{k=1}^n \braket{ \partial^{e_k} f | p_j \partial^{e_k} g}_2 \right| 
\\
& \leq \|p_{j}\|_{\infty} \sum_{k=1}^n \|\partial^{e_k} f\|_2
\|\partial^{e_k} g\|_2 \leq 
\|p_{j}\|_{\infty} \, \| f \| \, \| g \|
\end{align*}
for all $f, g \in W^1_{0,{\mathbb{C}}}({\Omega})$ and 
$j \in \{1,2\}$, it follows that 
$s_1$ and $s_2$ are bounded Hermitean positive sesquilinear forms. 
In addition, it follows that   
\begin{align} \label{s1pluss2posdef}
& s_1(f,f) + s_2(f,f) =  \braket{\,{\nabla_{w}} f \, | \, (p_1 + p_2) 
{\nabla_{w}} f \, }_{2,n} = 
\sum_{k=1}^n \braket{ \partial^{e_k} f | (p_1+p_2) \partial^{e_k} f}_2
\nonumber \\
& \geq \alpha  \sum_{k=1}^n \|\partial^{e_k} f\|_2^2 =
\alpha \, \|f\|^2
\end{align}
for every $f \in W^1_{0,{\mathbb{C}}}({\Omega})$, where $\alpha := 
\min\{\alpha_1,\alpha_2\}$. Further, it follows for 
$f \in X_1$ that  
\begin{equation*}
s_2(f,f) = \braket{\,{\nabla_{w}} f \, | \, p_2
{\nabla_{w}} f \, }_{2,n} \geq C_2 
\braket{\,{\nabla_{w}} f \, | \,
{\nabla_{w}} f \, }_{2,n} = C_2 \, \|f\|^2 
\end{equation*}
and 
\begin{equation*}
s_1(f,g) = \braket{\,{\nabla_{w}} f \, | \, p_1
{\nabla_{w}} g \, }_{2,n} = 0 
\end{equation*}
for every $g \in W^1_{0,{\mathbb{C}}}({\Omega})$. Further, 
we note that $s := s_1 + s_2$, as sum of two bounded 
Hermitean positive sesquilinear forms, is a bounded Hermitean 
positive sesquilinear form. In addition, as a consequence 
of (\ref{s1pluss2posdef}), $s$ is positive definite 
and hence a scalar product for $W^1_{0,{\mathbb{C}}}({\Omega})$. 
Also
\begin{align*}
& s(f,f) =  \braket{\,{\nabla_{w}} f \, | \, (p_1 + p_2) 
{\nabla_{w}} f \, }_{2,n} = 
\sum_{k=1}^n \braket{ \partial^{e_k} f | (p_1+p_2) \partial^{e_k} f}_2 
\\
& \leq \max\{\,\|p_1\|_{\infty}, \|p_2\|_{\infty}\} 
\sum_{k=1}^n \|\partial^{e_k} f\|_2^2 =
\max\{\,\|p_1\|_{\infty}, \|p_2\|_{\infty}\} \, \|f\|^2  \, \, , 
\end{align*}
for every $f \in W^1_{0,{\mathbb{C}}}({\Omega})$. Hence it follows 
by (\ref{s1pluss2posdef}) the equivalence of the norm that is induced 
on $W^1_{0,{\mathbb{C}}}({\Omega})$ by $s$ and 
$\|\,\|$. As a consequence, 
for $\omega \in L(W^1_{0,{\mathbb{C}}}({\Omega}),{\mathbb{C}})$, there 
is a unique $f \in W^1_{0,{\mathbb{C}}}({\Omega})$ such that 
\begin{equation*}
\omega = s_1(f,\cdot) + s_2(f,\cdot) \, \, . 
\end{equation*}
In particular, if $\ker \omega \supset X_1$, this implies that
\begin{align*}
& 0 = \omega(\hat{g}) = s_1(f,\hat{g}) + s_2(f,\hat{g}) = 
s_2(f,\hat{g}) = 
\braket{\,{\nabla_{w}} f \, | \, p_2 
{\nabla_{w}} \hat{g} \, }_{2,n} \\
& = \braket{\, f \, | \, 
{\nabla_{0}}^{*} p_2 
{\nabla_{w}} \hat{g} \, }_{2,n}
=
\braket{\,
(f|_{\Omega_2}) \, | \, {\nabla_{0,\Omega_2}}^{*}
(\,p_2|_{\Omega_2}) 
{\nabla_{w,\Omega_2}} g \, }_{2,n,\Omega_2}
\end{align*} 
for every $g \in W^{1}_{0, \mathbb{C}}({\Omega_2}) \cap 
W^{2}_{\mathbb{C}}({\Omega_2})$, where 
an index $\Omega_2$ indicates the association of structures 
to $\Omega_2$, instead of $\Omega$. Since, according to 
Theorem~\ref{regularcases},  
\begin{equation*}
\{\,{\nabla_{0,\Omega_2}}^{*} 
(\,p_2|_{\Omega_2}) 
{\nabla_{w,\Omega_2}} g \in L^{2}_{\mathbb{C}}({\Omega_2})
: g \in W^{1}_{0, \mathbb{C}}({\Omega_2}) \cap 
W^{2}_{\mathbb{C}}({\Omega_2}) \, 
\}
\end{equation*}
is dense in $L^{2}_{\mathbb{C}}({\Omega_2})$, the last implies that 
$f$ vanishes a.e. on $\Omega_2$ and hence that 
\begin{equation*}
\omega = s_1(f,\cdot) \, \, . 
\end{equation*}
In addition, if $g \in W^{1}_{0, \mathbb{C}}({\Omega})$ is such that 
\begin{equation*}
\omega = s_1(g,\cdot) \, \, ,
\end{equation*}
it follows that 
\begin{align*}
& 0 = 
\braket{\,{\nabla_{w}} (f-g) \, | \, p_1 
{\nabla_{w}} \hat{h} \, }_{2,n} = 
\braket{\, f-g \, | \, {\nabla_{0}}^{*} p_1 
{\nabla_{w}} \hat{h} \, }_{2,n} \\
& =
\braket{\,
(f-g)|_{\Omega_1} \, | \, {\nabla_{0,\Omega_1}}^{*}
(\,p_1|_{\Omega_1}) 
{\nabla_{w,\Omega_1}} h \, }_{2,n,\Omega_1}
\end{align*} 
for every $h \in W^{1}_{0, \mathbb{C}}({\Omega_1}) \cap 
W^{2}_{\mathbb{C}}({\Omega_1})$, where 
an index $\Omega_1$ indicates the association of structures 
to $\Omega_1$, instead of $\Omega$. Since, according to 
Theorem~\ref{regularcases},   
\begin{equation*}
\{\,{\nabla_{0,\Omega_1}}^{*} 
(\,p_1|_{\Omega_1}) 
{\nabla_{w,\Omega_1}} h \in L^{2}_{\mathbb{C}}({\Omega_1})
: h \in W^{1}_{0, \mathbb{C}}({\Omega_1}) \cap 
W^{2}_{\mathbb{C}}({\Omega_1}) \, 
\}
\end{equation*}
is dense in $L^{2}_{\mathbb{C}}({\Omega_1})$, the last implies that 
$f - g$ vanishes a.e. on $\Omega_1$ and hence that 
$f - g \in X_1$.
\end{proof}

We give a concrete example of the application of Lions'
Lemma~\ref{basiclemmalions} to the diffusion
equation~\eqref{diffusionEquation2}.

\begin{corollary}
Let $f \in W^{1}_{0, \mathbb{C}}({\Omega})$, $\varepsilon > 0$,
$k \in {\mathbb{N}} \cup \{-1\}$, and $X_1$ as in
Theorem~\ref{thm:generalizationLions}. Restriction of a function to
$\Omega_i$ is indicated by an addition of an index $i$.
\begin{itemize}
\item[(i)] There is a unique $u_{\varepsilon} \in W^{1}_{0, \mathbb{C}}({\Omega})$
such that 
\begin{equation*}
A u_{\varepsilon} = f \,.
\end{equation*} 

\item[(ii)] There is $C > 0$ such that 
\begin{equation*}
\left\vvvert \sum_{j=-1}^{k} \varepsilon^{j} u_{j} - u_{\varepsilon} \, 
\right\vvvert_{1} \leq C \varepsilon^{k+1}  \, \, , 
\end{equation*} 
where $u_{-1} \in X_{1}$ and $u_{0},\dots,u_{k} \in W^{1}_{0, \mathbb{C}}({\Omega})$
are uniquely determined by 

$$
\begin{array}{lllllll}
- {\nabla_{0}}^{*} p_{2 2} {\nabla_{w}} u_{-1 2} & = & f_2 \, , & & \\
- {\nabla_{0}}^{*} p_{1 1} {\nabla_{w}} u_{0 1} & = & 0, &
\left( \frac{\partial u_{0 1}}{\partial \nu} - 
\frac{\partial u_{-1 2}}{\partial \nu} \right) \bigg|_{\partial \Omega_2} & = & 0, 
& u_{0 2} = 0 \, ,\\
- {\nabla_{0}}^{*} p_{1 1} {\nabla_{w}} u_{j 1} & = & 0, & 
\left( \frac{\partial u_{j 1}}{\partial \nu} - 
\frac{\partial u_{(j-1)2}}{\partial \nu} \right) \bigg|_{\partial \Omega_2} & = & 
0\, , & \\
- {\nabla_{0}}^{*} p_{2 2} {\nabla_{w}} u_{j 2} & = & 0,&
\left( u_{j 2} - u_{(j-1) 1} \right) \big|_{\partial \Omega_2} & = & 0\, , & 
\end{array}
$$
where $j \in \{1,\dots,k\}$.
\end{itemize}
\end{corollary}

\section{Concluding remarks}

Based on the foundation provided by our previous
paper~\cite{AkBe2008}, in this paper, we generalize Lions' results in
various ways.  Our results provide the existence of strong solutions
of the operator equation instead of Lions' weak solutions. In
particular, we generalize Lions' results to include diffusivities that
are piecewise $C^1$ up to the boundary of $\Omega_1$ and $\Omega_2$,
where $\Omega_2 := \Omega \, \setminus \, \overline{\Omega}_1$.  Note
that the geometric configuration is restricted to the case that the
boundaries of $\Omega_1$ and $\Omega$ have a non-empty intersection.
In the one dimensional case, a full characterization of the limiting
inverse operator is given in our preceding paper~\cite{AkBe2008},
independent of the configuration.  The other case corresponding to
the right of Figure~\ref{fig:domain1}, i.e., when the boundary of
$\Omega_1$ has an empty intersection with that of $\Omega$ with
$n \geq 2$, is still largely open.  On the other hand, for that
configuration, a characterization of the limit of the discretized
inverse operators with piecewise constant coefficients is given by the
first author in~\cite{AkYe2009} using linear finite element and finite
volume methods.

\section{Appendix}


The following are the assumptions for Lions' abstract Lemma.

\begin{ass} \label{lionsass}
Let $X$ be a non-trivial complex Hilbert space, $X_1$ a closed 
subspace 
of $X$ and 
$s_1 : X^2 \rightarrow 
{\mathbb{C}}$, $s_2 : X^2 \rightarrow {\mathbb{C}}$ be bounded 
sesquilinear forms on $X$, i.e., 
sesquilinear forms for which there are 
$C_1,C_2 \geq 0$ such that 
\begin{equation*}
|s_i(\xi,\eta)| \leq C_i \, \|\xi\| \, \|\eta\|
\end{equation*}  
for all $\xi,\eta \in X$ and $i \in \{1,2\}$. In addition, 
let $s_1, s_2$ be Hermitean, positive and satisfy 
the following conditions.
\begin{itemize}
\item[(i)] There is $\alpha > 0$ such that 
\begin{equation*}
s_1(\xi,\xi) + s_2(\xi,\xi) \geq \alpha \|\xi\|^2
\end{equation*}
for all $\xi \in X$,
\item[(ii)] 
\begin{itemize}
\item[1)] $s_1(\xi,\xi^{\prime}) = 0$ for all $\xi \in X_1$ and 
$\xi^{\prime} \in X$,
\item[2)]  for every 
$\omega \in L(X,{\mathbb{C}})$ such that $\ker \omega \supset X_1$, 
there is $\xi \in X$ such that $\omega = s_1(\xi,\cdot)$. In addition,
if $\xi^{\prime} \in X$ is such that $\omega = s_1({\xi}^{\prime},
\cdot)$, then ${\xi}^{\prime} - \xi \in X_1$. 
\end{itemize}
\item[(iii)] There is $\alpha_2 > 0$ such that 
\begin{equation*}
s_2(\xi,\xi) \geq \alpha_2 \|\xi\|^2
\end{equation*}
for all $\xi \in X_1$.
\end{itemize}
\end{ass}

\begin{lemma} \label{basiclemmalions}
Assume~\ref{lionsass}.  Then, for every $\omega \in L(X,{\mathbb{C}})$
such that $\ker \omega \supset X_1$, there is a unique $\xi \in X$
such that $\omega = s_1(\xi,\cdot)$ and $s_2(\xi,\xi^{\prime}) = 0$
for all $\xi^{\prime} \in X_1$.
\end{lemma}

\begin{proof}
We note that, since $s_2$ is sesquilinear, Hermitean and 
positive, there is a uniquely determined positive 
self-adjoint $T_2 \in L(X,X)$ such that 
\begin{equation*}
s_2(\xi,{\xi}^{\prime}) = \braket{\xi|T_2 {\xi}^{\prime}}
\end{equation*}
for all $\xi, {\xi}^{\prime} \in X$. 
In the following, we denote by $P_1$ the projection onto $X_1$. Then 
the restriction  
$T_{21}$ of $P_1 T_2 P_1$ in domain and in image to $X_1$ is a positive 
self-adjoint element of $L(X_1,X_1)$. Further, since there is 
$\alpha_2 > 0 $ such that for every 
$\xi \in X_1$ 
\begin{equation*}
s_2(\xi,\xi) = \braket{P_1 \xi|T_2 P_1 \xi} = 
\braket{\xi| T_{21} \xi} \geq \alpha_2 \|\xi\|^2 \, \, ,   
\end{equation*} 
$T_{21}$ is strictly positive and hence bijective. If $\omega$ is an 
element of 
$L(X,{\mathbb{C}})$  such that $\ker \omega \supset X_1$, 
then there is $\xi \in X$ such that 
$\omega = s_1(\xi,\cdot)$. Also for $\xi^{\prime \prime} \in X_1$,
$\omega = s_1(\xi + \xi^{\prime \prime},\cdot)$.
Then 
\begin{equation} \label{addeq}
s_2(\xi + \xi^{\prime \prime},\xi^{\prime }) = 0 
\end{equation}
for all $\xi^{\prime} \in X_1$ if and only if 
\begin{equation*}
\braket{T_{21} \xi^{\prime \prime}| \xi^{\prime}} = 
s_2(\xi^{\prime \prime},\xi^{\prime}) = - s_2(\xi,\xi^{\prime})
= - \braket{\xi|T_2 \xi^{\prime}} = 
- \braket{P_1 T_2 \xi|\xi^{\prime}} 
\end{equation*}
for all $\xi^{\prime} \in X_1$. Hence,
(\ref{addeq}) is satisfied 
for all $\xi^{\prime} \in X_1$ if 
\begin{equation*}
 \xi^{\prime \prime} = - T_{21}^{-1} P_1 T_2 \xi \, \, .
\end{equation*}
Further, if $\xi_1,\xi_2 \in X$ are such that 
$\omega = s_1(\xi_i,\cdot)$ and 
$s_2(\xi_i,\xi^{\prime}) = 0$ for all $\xi^{\prime} \in X_1$
and $i \in \{1,2\}$, then $\xi_{1} - \xi_2 \in X_1$ and 
\begin{equation*}
0 = s_2(\xi_{1} - \xi_2,\xi_{1} - \xi_2) \geq \alpha_2 
\|\xi_{1} - \xi_2\|^2  \, \, . 
\end{equation*}
Hence $\xi_1 = \xi_2$.
\end{proof}

\begin{lemma} \label{lemma:LionsAbstract}
Assume~\ref{lionsass}. For $\eta \in X$, $\varepsilon > 0$
and $k \in {\mathbb{N}} \cup \{-1\}$, it follows that
\begin{itemize}
\item[(i)] There is a unique ${\xi}_{\varepsilon} \in X$ such that 
\begin{equation*}
s_1({\xi},\xi_{\varepsilon}) 
+ \varepsilon \, s_2({\xi}, 
\xi_{\varepsilon}) 
= \braket{{\xi}|\eta}
\end{equation*} 
for all $\xi \in X$.
\item[(ii)] There is $C > 0$ such that 
\begin{equation*}
\left\|\, \sum_{j=-1}^{k} \varepsilon^{j} {\xi}_{j} - 
\xi_{\varepsilon}
\, \right\|
\leq C \varepsilon^{k+1}  \, \, , 
\end{equation*} 
where $\xi_{-1} \in X_{1}$ and $\xi_{0},\dots,\xi_{k} \in X$ are 
uniquely determined by 
\begin{align*}
& s_2(\xi,\xi_{-1}) = \braket{\xi|\eta} 
\, \, \textrm{for all} \, \, \xi \in X_1\, \,  , \\
& s_1(\xi,\xi_{0}) = \braket{\xi|\eta} - s_2(\xi,\xi_{-1})
\, \, \textrm{for all} \, \, \xi \in X 
\, \, , \, \, 
s_2(\xi,\xi_{0}) = 0 \, \, \textrm{for all} \, \, \xi \in X_1\, \, 
, \\
& s_1(\xi,\xi_{j}) = - s_2(\xi,\xi_{j-1})
\, \, \textrm{for all} \, \, \xi \in X 
\, \, , \, \, 
s_2(\xi,\xi_{j}) = 0 \, \, \textrm{for all} \, \, \xi \in X_1\, \, , 
\end{align*}
where $j \in \{1,\dots,k\}$.
\end{itemize}
\end{lemma}

\begin{proof}
`(i)': Since $s_1,s_2$ are sesquilinear, Hermitean and 
positive, there are uniquely determined positive 
self-adjoint $T_1,T_2 \in L(X,X)$ such that 
\begin{equation*}
s_i(\xi,{\xi}^{\prime}) = \braket{\xi|T_i {\xi}^{\prime}}
\end{equation*}
for all $\xi, {\xi}^{\prime} \in X$ and $i \in \{1,2\}$. Hence 
\begin{equation*}
s_1(\xi,{\xi}^{\prime}) 
+ \varepsilon \, s_2(\xi,{\xi}^{\prime}) =
\braket{\xi|(T_1 + \varepsilon \, T_2) {\xi}^{\prime}}
\end{equation*}
for all $\xi, {\xi}^{\prime} \in X$. In addition,  
there is $\alpha > 0$ such that 
\begin{equation*}
s_1(\xi,{\xi}) 
+ \varepsilon \, s_2(\xi,{\xi}) = 
\braket{\xi|(T_1 + \varepsilon \, T_2) {\xi}} \geq 
\alpha \, \|\xi\|^2
\end{equation*}
for all $\xi \in X$. As a consequence, 
$T_1 + \varepsilon \, T_2$ is strictly 
positive and hence bijective. Therefore 
\begin{equation*}
{\xi}_{\varepsilon} := (T_1 + \varepsilon \, T_2)^{-1} \eta 
\end{equation*}
satisfies 
\begin{equation*}
s_1({\xi},\xi_{\varepsilon}) 
+ \varepsilon \, s_2({\xi}, 
\xi_{\varepsilon}) 
= \braket{{\xi}|\eta}
\end{equation*} 
for all $\xi \in X$. Further, if ${\xi}^{\prime} \in X$ is such 
that 
\begin{equation*}
s_1({\xi},\xi^{\prime}) 
+ \varepsilon \, s_2({\xi}, 
\xi^{\prime}) 
= \braket{{\xi}|\eta}
\end{equation*} 
for all $\xi \in X$, then 
\begin{equation*}
0 = s_1(\xi^{\prime} - \xi_{\varepsilon},\xi^{\prime} - \xi_{\varepsilon}) 
+ \varepsilon \, s_2(\xi^{\prime} - \xi_{\varepsilon}, 
\xi^{\prime} - \xi_{\varepsilon}) \geq 
\alpha \, \|\xi^{\prime} - \xi_{\varepsilon}\|^2 
\end{equation*} 
and hence $\xi^{\prime} = \xi_{\varepsilon}$.
\newline
`(ii)':
For this, let $\xi_{-1},\dots,\xi_{k} \in X$ and 
\begin{equation*}
\xi_{\varepsilon k} := \sum_{j=-1}^{k} \varepsilon^{j} 
\xi_{j} \, \, . 
\end{equation*}
Then 
\begin{align*}
& s_1(\xi,\xi_{\varepsilon k}) 
+ \varepsilon \, s_2(\xi,\xi_{\varepsilon k}) = 
\sum_{j=-1}^{k} \varepsilon^{j} s_1(\xi,\xi_{j}) + 
\sum_{j=-1}^{k} \varepsilon^{j+1} s_2(\xi,\xi_{j}) \\
& = \varepsilon^{-1} s_1(\xi,\xi_{-1}) +  
\varepsilon^{k+1} s_2(\xi,\xi_{k}) + 
\sum_{j=0}^{k} \varepsilon^{j}  \, [ \, s_1(\xi,\xi_{j}) + 
s_2(\xi,\xi_{j-1}) \, ] 
\end{align*}
for every $\xi \in X$.
Therefore, if 
\begin{equation*}
s_1(\xi,\xi_{-1}) = 0 \, \, , \, \, 
s_1(\xi,\xi_{0}) = \braket{{\xi}|\eta} - s_2(\xi,\xi_{-1}) \, \, , \, \, 
s_1(\xi,\xi_{j}) = - s_2(\xi,\xi_{j-1}) \, \, , 
\end{equation*}
for all $\xi \in X$,
where $j \in \{1,\dots,k\}$, then  
\begin{equation*}
s_1(\xi,\xi_{\varepsilon k}) 
+ \varepsilon \, s_2(\xi,\xi_{\varepsilon k}) = 
\braket{{\xi}|\eta} + \varepsilon^{k+1} s_2(\xi,\xi_{k})
\end{equation*}
for all $\xi \in X$
and hence 
\begin{equation*}
s_1(\xi,\xi_{\varepsilon k} - \xi_{\varepsilon}) 
+ \varepsilon \, s_2(\xi,\xi_{\varepsilon k}-\xi_{\varepsilon}) = 
\varepsilon^{k+1} s_2(\xi,\xi_{k})
\end{equation*}
for all $\xi \in X$. In particular, this implies that 
\begin{align*}
& \alpha \, \|\xi_{\varepsilon k} - \xi_{\varepsilon}\|^2 
\leq s_1(\xi_{\varepsilon k} - \xi_{\varepsilon},
\xi_{\varepsilon k} - \xi_{\varepsilon}) 
+ \varepsilon \, 
s_2(\xi_{\varepsilon k} - \xi_{\varepsilon},\xi_{\varepsilon k}-\xi_{\varepsilon}) \\
& =
\varepsilon^{k+1} s_2(\xi_{\varepsilon k} - \xi_{\varepsilon},\xi_{k}) \leq C_2 \, \varepsilon^{k+1} \, 
\|\xi_{k}\| \, \|\xi_{\varepsilon k} - \xi_{\varepsilon}\| \, \, , 
\end{align*}
where $C_2 \geq 0$ is such that 
\begin{equation*}
|s_2(\xi,\eta)| \leq C_2 \, \|\xi\| \, \|\eta\|
\end{equation*}  
for all $\xi,\eta \in X$. Hence it follows that 
\begin{equation*}
\|\xi_{\varepsilon k} - \xi_{\varepsilon}\| \leq \frac{C_2}{\alpha}
\, \|\xi_{k}\| \, \varepsilon^{k+1} \, \, .
\end{equation*}
In the following, we denote by $P_1$ the projection onto $X_1$. Then 
the restriction  
$T_{21}$ of $P_1 T_2 P_1$ in domain and in image to $X_1$ is a positive 
self-adjoint element of $L(X_1,X_1)$. Further, since there is 
$\alpha_2 > 0 $ such that for every 
$\xi \in X_1$ 
\begin{equation*}
s_2(\xi,\xi) = \braket{P_1 \xi|T_2 P_1 \xi} = 
\braket{\xi| T_{21} \xi} \geq \alpha_2 \|\xi\|^2 \, \, ,   
\end{equation*} 
$T_{21}$ is strictly positive and hence bijective. Hence 
it follows for $\xi \in X_1$ and 
\begin{equation*}
\xi_{-1} := T_{21}^{-1}P_1 \eta
\end{equation*}
that 
\begin{align*}
& s_2(\xi,\xi_{-1}) =  
\braket{\xi|T_2 T_{21}^{-1}P_1 \eta} = 
\braket{P_1 \xi|T_2 T_{21}^{-1}P_1 \eta} =
\braket{\xi|P_1 T_2 P_1 T_{21}^{-1}P_1 \eta} \\
& = 
\braket{\xi|P_1 \eta} = \braket{P_1 \xi|\eta} = \braket{\xi|\eta}
\, \, . 
\end{align*}
Further, if $\xi_{-1}^{\prime} \in X_1$ is such that 
\begin{equation*}
s_2(\xi,\xi_{-1}^{\prime}) = \braket{\xi|\eta}
\end{equation*}
for every $\xi \in X_1$, then 
\begin{equation*}
0 = s_2(\xi_{-1}^{\prime} - \xi_{-1},\xi_{-1}^{\prime} - \xi_{-1})
\geq \alpha_2 \|\xi_{-1}^{\prime} - \xi_{-1}\|^2
\end{equation*}
and hence $\xi_{-1}^{\prime} = \xi_{-1}$. 
Further, since 
\begin{equation*}
\braket{\xi|\eta} - s_{2}(\xi,\xi_{-1}) = 0 
\end{equation*}
for every $\xi \in X_1$, it follows that 
\begin{equation*}
\braket{\eta|\cdot} - s_{2}(\xi_{-1},\cdot) 
\end{equation*}
is an element of $L(X,{\mathbb{C}})$ whose kernel contains $X_1$. 
Hence, by Lemma~\ref{basiclemmalions}, there is 
a unique $\xi_{0} \in X$ such that
\begin{equation*}
s_1(\xi,\xi_{0}) = \braket{\xi|\eta} - s_{2}(\xi,\xi_{-1})
\end{equation*}
for every $\xi \in X$ and $s_2(\xi,\xi_0) = 0$ for every 
$\xi \in X_1$. Finally,  by Lemma~\ref{basiclemmalions}, it follows 
recursively the existence and 
uniqueness of $\xi_1,\dots,\xi_{k} \in X$ such that 
\begin{equation*}
s_1(\xi,\xi_{j}) = - s_2(\xi,\xi_{j-1})
\, \, \textrm{for all} \, \, \xi \in X 
\, \, , \, \, 
s_2(\xi,\xi_{j}) = 0 \, \, \textrm{for all} \, \, \xi \in X_1\, \, , 
\end{equation*}
for $j \in \{1,\dots,k\}$. 
\end{proof}

\end{document}